\documentclass[12pt]{amsart}

\usepackage[margin=1in]{geometry}
\usepackage{amsmath}
\usepackage{amssymb}

\newtheorem{theorem}{Theorem}[section]
\newtheorem{lemma}[theorem]{Lemma}
\newtheorem{proposition}[theorem]{Proposition}
\newtheorem{corollary}[theorem]{Corollary}

\theoremstyle{definition}
\newtheorem{definition}[theorem]{Definition}

\DeclareMathOperator{\Sym}{Sym}
\DeclareMathOperator{\Aut}{Aut}
\newcommand{\dcup}{\mathbin{\dot\cup}}
\newcommand{\cB}{\mathcal{B}}

\begin{document}

\title{Chouinard's conjecture for graphical $t$-designs}

\author{Yeow Meng Chee} \address{Singapore University of Technology and
Design, Singapore} \email{ymchee@sutd.edu.sg}


\begin{abstract}
A $t$-wise balanced design on the edge set of a complete graph is graphical
if its block multiset is invariant under the induced action of the symmetric
group on the vertices. Chouinard conjectured that, for each fixed index
$\lambda$, there are only finitely many nontrivial simple graphical $t$-wise
balanced designs with $t > 1$. We prove the conjecture for $t$-designs,
which are the $t$-wise balanced designs whose blocks all have the same size.
Our theorem does not require simplicity, and we bound all parameters of the
designs by explicit polynomials in $\lambda$.
\end{abstract}

\maketitle

\section{Introduction}

A $t$-wise balanced design of index $\lambda$ is a pair $(X, \cB)$, where
$X$ is a finite set and $\cB$ is a multiset of subsets of $X$, called
\emph{blocks}, such that every $t$-subset of $X$ lies in exactly $\lambda$
blocks, counted with multiplicity. The design $(X, \cB)$ is a
\emph{$t$-design} if all of its blocks have the same size, and it is
\emph{simple} if no block is repeated. Throughout the paper, $V$ is an
$n$-element set, $K_n$ is the complete graph with vertex set $V$, and
\[
X = E(K_n) = \binom{V}{2}
\]
is its edge set, which is the set of all $2$-subsets of $V$. The symmetric
group $\Sym(V)$ acts naturally on $X$. A design on $X$ is called
\emph{graphical} if its block multiset is invariant under this action.
Writing $\mu_{\cB}(B)$ for the multiplicity of a subset $B$ of $X$ in the
block multiset $\cB$, a design on $X$ is graphical precisely when
$\mu_{\cB}(B)$ depends only on the isomorphism type of the spanning subgraph
$(V, B)$ of $K_n$.

Graphical $t$-designs first appeared in the work of Alltop
\cite{Alltop1966}. His note proved what is now called Alltop's lemma, which
counts the spanning subgraphs of $K_n$ isomorphic to a given graph and
containing a fixed subgraph, in terms of the orders of the automorphism
groups of these two graphs. He used it to show that, for every odd $n \geq
3$, the edge sets of the spanning subgraphs of $K_n$ consisting of a single
cycle of length $(n+3)/2$ together with $(n-3)/2$ isolated vertices form a
graphical $2$-design. A graph whose isomorphism class forms a graphical
$t$-design is called \emph{$t$-edge-balanced}. Further $2$-edge-balanced
graphs appear in \cite{Caliskan2016, CaliskanChee2014}. The first
$3$-edge-balanced graphs are constructed in \cite{Chee2026}, where it is
also shown that no nontrivial $t$-edge-balanced graphs exist for $t \geq 4$.
Klin \cite{Klin1970} described graphical $t$-designs independently around
1970 in his study of the overgroups of the symmetric group acting on pairs.
Kramer and Mesner \cite{KramerMesner1976} determined all simple graphical
$t$-designs with $n \leq 6$ in their study of $t$-designs with prescribed
automorphism groups. They attribute to R.~M.~Wilson an early example, namely
a graphical $3$-$(10,4,1)$ design on $E(K_5)$. The general notion of a
graphical $t$-wise balanced design is due to Chouinard, Kramer, and Kreher
\cite{CKK1983}. They determined all simple graphical $t$-wise balanced
designs of index one and two, using derivation with respect to full stars as
an infinite descent. These results led Chouinard to conjecture that, for
every positive integer $\lambda$, there are only finitely many nontrivial
simple graphical $t$-wise balanced designs with $t > 1$. The conjecture
dates to around the time of \cite{CKK1983}, but it circulated only
privately. The author learned of it from Kramer in early 1989 and recorded
it in \cite{Chee1991}, which was its first appearance in print. It is also
stated in \cite{Chouinard1996} and \cite[Conjecture 1.2]{CheeKaski2008}.

Subsequent work established finiteness under additional restrictions on the
parameters, mostly for simple designs. Chouinard \cite{Chouinard1996} proved
that a nontrivial graphical $t$-wise balanced design with $t > 1$ satisfies
\[
n < 2t + \lambda + 4.
\]
In particular, there are only finitely many such designs for each fixed $t$
and $\lambda$. Chee \cite{Chee1991} determined the graphical $t$-designs
with block sizes three and four, introducing the approach through
Kramer--Mesner matrices with entries polynomial in $n$. By the same method,
aided by symbolic computation, Chee \cite{Chee1992} proved the nonexistence
of nontrivial graphical quintuple systems, and Mori \cite{Mori1994} proved
the nonexistence of nontrivial graphical $\kappa$-tuple systems for $6 \leq
\kappa \leq 8$, as well as for $\kappa$ at most the number of vertices. Both
results address a conjecture of Chee \cite{Chee1992} that no nontrivial
graphical $k$-$\bigl(\binom{n}{2}, k + 1, \lambda\bigr)$ design exists for
$k \geq 5$, which remains open. Betten, Klin, Laue, and Wassermann
\cite{BKLW1999} developed this method further under the name of polynomial
Kramer--Mesner matrices, with the entries computed by Alltop's lemma, and
proved finiteness for block size $k = t + 1$. Chee \cite{Chee2001} extended
their result to $k \leq 4t/3$. Chee and Kaski \cite{CheeKaski2008}
enumerated graphical $t$-designs in a further range of parameters and noted
that Chouinard's conjecture remains open.

In all of this work, the parameters are bounded in terms of $\lambda$ only
under additional restrictions. Chouinard's inequality bounds $n$ in terms of
$t$ and $\lambda$ but leaves $t$ free, and the method based on the growth of
the polynomial entries of the Kramer--Mesner matrix requires the block size
to be a bounded multiple of the strength. We remove both restrictions. The
Ray--Chaudhuri--Wilson theorem bounds the block size of a nontrivial
graphical $t$-design with $t \leq n$ by $O(t^{3/2})$, with no assumption on
$k$, and for a test graph $T$ with components $K_1$, $K_2$, and $K_4$, whose
automorphism group has only orbits of quadratic size on the missing edges,
the orbit of any block containing $E(T)$ under the automorphism group of $T$
has cardinality $\Omega(t^2/(k - t))$, while the design condition permits at
most $\lambda$ such blocks. For large $t$ these estimates collide, and the
full-star descent of \cite{CKK1983} carries the resulting bound on $t$ to
all nontrivial graphical $t$-designs.

\begin{theorem}
\label{thm:main}
For every positive integer $\lambda$, there are only finitely many
nontrivial graphical $t$-$\bigl(\binom{n}{2}, k, \lambda\bigr)$ designs with
$t \geq 2$.
\end{theorem}

Here, \emph{nontrivial} means that $t < k < \binom{n}{2}$ and the $t$-design
neither is empty nor contains the complete design. A refinement of the
argument yields explicit bounds, polynomial in $\lambda$, on all the
parameters (Theorem~\ref{thm:explicit}). The general case of Chouinard's
conjecture, for $t$-wise balanced designs, remains open. Precise definitions
are given in Section~\ref{sec:preliminaries}.

\section{Preliminaries}
\label{sec:preliminaries}

Throughout, $\lambda$ is a fixed positive integer, and the constants implied
by the $O$, $\Omega$, and $\Theta$ notation may depend on $\lambda$.

Let $t$ be a positive integer, let $X$ be a set of cardinality $v$, and let
$K$ be a set of integers, each at least $t$. A \emph{$t$-wise balanced
design} of type $t$-$(v, K, \lambda)$ is a pair $(X, \cB)$, where $\cB$ is a
multiset of subsets of $X$ with cardinalities in $K$, such that every
$t$-subset of $X$ is contained in exactly $\lambda$ members of $\cB$,
counted with multiplicity. A $t$-wise balanced design of type $t$-$(v,
\{k\}, \lambda)$, whose members are then all $k$-subsets, is called a
\emph{$t$-design}, and we abbreviate its type to $t$-$(v, k, \lambda)$. We
call $t$ the \emph{strength}, $k$ the \emph{block size}, and $\lambda$ the
\emph{index} of a $t$-$(v, k, \lambda)$ design.

The multiplicity of a $k$-subset $B$ of $X$ in $\cB$ is denoted
$\mu_{\cB}(B)$, and the $k$-subsets $B$ with $\mu_{\cB}(B) \geq 1$ are the
\emph{blocks}. A $t$-design is \emph{simple} if $\mu_{\cB}(B) \leq 1$,
\emph{complete} if $\mu_{\cB}(B) = 1$, and \emph{empty} if $\mu_{\cB}(B) =
0$, in each case for every $k$-subset $B$ of $X$. A $t$-design
\emph{contains the complete design} if $\mu_{\cB}(B) \geq 1$ for every
$k$-subset $B$ of $X$, and it is \emph{nontrivial} if $t < k < v$ and it
neither is empty nor contains the complete design.

It is well known (see, for example, \cite{BJL1999}) that the number of
blocks of a $t$-$(v, k, \lambda)$ design, counted with multiplicity, is
\begin{equation}
\label{eq:blocks}
b = |\cB| = \lambda \frac{\binom{v}{t}}{\binom{k}{t}}.
\end{equation}
For a graph $G$, we write $E(G)$ for its edge set. The group $\Sym(V)$ of
all permutations of $V$ acts on $X = E(K_n)$ by $\sigma(\{x, y\}) =
\{\sigma(x), \sigma(y)\}$, and on subsets of $X$ elementwise. For a spanning
subgraph $T$ of $K_n$, let $\Aut(T)$ denote its full automorphism group, the
group of all $\sigma \in \Sym(V)$ with $\sigma(E(T)) = E(T)$. We write $G
\dcup H$ for the disjoint union of the graphs $G$ and $H$, and $mG$ for the
disjoint union of $m$ copies of $G$. The \emph{full star} at a vertex $x \in
V$ is
\[
S_x = \bigl\{ \{x, y\} : y \in V \setminus \{x\} \bigr\},
\]
the set of the $n - 1$ edges of $K_n$ through $x$. Given a $t$-design on $X
= E(K_n)$, write
\[
q = \binom{n}{2} - t \qquad \text{and} \qquad d = k - t.
\]
We call $d$ the \emph{excess}, since a block that contains the edge set of a
spanning $t$-edge subgraph $T$ consists of $E(T)$ together with exactly $d$
edges of $K_n$ outside $T$. Similarly, $q$ is the number of edges of $K_n$
outside $T$, from which the $d$ excess edges are drawn. The $t$-designs of
interest are those invariant under the action of $\Sym(V)$.

\begin{definition}
\label{def:graphical}
A $t$-$\bigl(\binom{n}{2}, k, \lambda\bigr)$ design $(X, \cB)$ is
\emph{graphical} if
\[
\mu_{\cB}(\sigma(B)) = \mu_{\cB}(B)
\]
for all $\sigma \in \Sym(V)$ and all $B \in \binom{X}{k}$.
\end{definition}

This is the usual invariance of the block multiset under the induced action
of $\Sym(V)$: the image of $\cB$ under $\sigma$ is the multiset with
multiplicity function $\mu_{\cB} \circ \sigma^{-1}$, so the equalities
$\mu_{\cB}(\sigma(B)) = \mu_{\cB}(B)$, for all $B \in \binom{X}{k}$, hold if
and only if the multiset $\cB$ satisfies $\sigma(\cB) = \cB$. We use the
following form of the Ray--Chaudhuri--Wilson theorem. For even $t$, it is
Theorem 1 of \cite{RCW1975}, which allows repeated blocks and bounds even
the number of distinct blocks. For odd $t$, it follows from the corollary to
\cite[Theorem 1]{RCW1975}, stated immediately after that theorem.

\begin{theorem}[Ray--Chaudhuri--Wilson]
\label{thm:rcw}
Let $(X, \cB)$ be a $t$-$(v, k, \lambda)$ design with $t \geq 2$, and put $s
= \lfloor t/2 \rfloor$. If $v > k + s$, then $b \geq \binom{v}{s}$.
\end{theorem}

We also use Chouinard's bound, which is the specialization to $t$-designs of
\cite[Theorem 2.1]{Chouinard1996}.

\begin{theorem}[Chouinard]
\label{thm:chouinard}
Let $(X, \cB)$ be a nontrivial graphical $t$-$\bigl(\binom{n}{2}, k,
\lambda\bigr)$ design with $t > 1$. Then $n < 2t + \lambda + 4$.
\end{theorem}

\section{Reduction by Full-Star Derivation}
\label{sec:reduction}

This section carries out the reduction to the range $t \leq n$. We first
record an elementary restriction on the parameters of a nontrivial
$t$-design. The result goes back to Wilson \cite{Wilson1973}, as noted in
the remark following Theorem~1 of \cite{GLL1980}, and we include a short
proof for completeness.

\begin{lemma}
\label{lem:vkt}
Let $(X, \cB)$ be a nontrivial $t$-$(v, k, \lambda)$ design. Then $v > k +
t$.
\end{lemma}

\begin{proof}
Suppose instead that $v \leq k + t$, and put $\ell = v - k$. Then $0 < \ell
\leq t$. For $0 \leq i \leq t$, the number $\lambda_i$ of members of $\cB$
containing a fixed $i$-subset $I$ of $X$, counted with multiplicity, does
not depend on $I$. Indeed, counting the pairs $(T_0, B)$, where $B$ is a
member of $\cB$ and $T_0$ is a $t$-subset with $I \subseteq T_0 \subseteq
B$, in two ways gives
\[
\lambda_i \binom{k - i}{t - i} = \lambda \binom{v - i}{t - i}.
\]
By inclusion--exclusion, the number of members of $\cB$ disjoint from a
fixed $\ell$-subset $J$ of $X$, counted with multiplicity, is
\[
\sum_{i=0}^{\ell} (-1)^i \binom{\ell}{i} \lambda_i,
\]
which does not depend on $J$ either. Summing this number over all
$\ell$-subsets $J$ counts each member $B$ of $\cB$ exactly once, since the
unique $\ell$-subset disjoint from $B$ is $X \setminus B$. The number of
members of $\cB$ disjoint from a fixed $\ell$-subset $J$ is therefore $b /
\binom{v}{\ell}$, which is positive because a nontrivial $t$-design is not
empty. The only $k$-subset of $X$ disjoint from $J$ is $X \setminus J$, so
this number equals $\mu_{\cB}(X \setminus J)$. Hence $\mu_{\cB}(X \setminus
J) \geq 1$ for every $\ell$-subset $J$ of $X$, and every $k$-subset of $X$
has the form $X \setminus J$ for such a $J$, so $(X, \cB)$ contains the
complete design. This contradicts nontriviality.
\end{proof}

For a nontrivial $t$-design on $X = E(K_n)$, Lemma~\ref{lem:vkt} gives
\begin{equation}
\label{eq:qkt}
q > k > t \qquad \text{and} \qquad 0 < d < q.
\end{equation}
Derivation with respect to a full star removes a vertex. Let $x \in V$. The
members of the derived design are the sets $B \setminus S_x$, where $B$ runs
over the members of $\cB$ that contain $S_x$, counted with multiplicity.
These sets consist of edges of the complete graph on $V \setminus \{x\}$.
This derivation is due to Chouinard, Kramer, and Kreher \cite[Theorem
2.9]{CKK1983} for simple designs. The following lemma extends it to designs
with repeated blocks and adds that the derivation preserves the two
quantities $q$ and $d$. Since $q$ is unchanged under repeated derivation, a
bound on the strength of the last design so obtained yields, through $t <
q$, a bound on the strength of the original design.

\begin{lemma}[Full-star derivation]
\label{lem:derivation}
Let $(X, \cB)$ be a nontrivial graphical $t$-$\bigl(\binom{n}{2}, k,
\lambda\bigr)$ design with $t \geq n + 1$. Then there is a graphical
$t'$-$\bigl(\binom{n-1}{2}, k', \lambda\bigr)$ design, where
\[
t' = t - (n - 1)
\qquad \text{and} \qquad
k' = k - (n - 1).
\]
These parameters satisfy
\begin{equation}
\label{eq:invariants}
\binom{n - 1}{2} - t' = \binom{n}{2} - t
\qquad \text{and} \qquad
k' - t' = k - t.
\end{equation}
Moreover, we have $t' < k' < \binom{n - 1}{2}$, so the derived design either
is nontrivial or contains the complete design. If it contains the complete
design, then $\lambda \geq \binom{q}{d} \geq q > t$.
\end{lemma}

\begin{proof}
Let $x \in V$, and consider the full star $S_x$, with $|S_x| = n - 1 \leq
t$. The edges outside $S_x$ are those of $K_{V \setminus \{x\}}$. Identify
them with the edge set of $K_{n-1}$. For each $k'$-subset $C$ of $X
\setminus S_x$, put
\[
\mu_{\cB'}(C) = \mu_{\cB}(C \cup S_x),
\]
which defines a multiset $\cB'$ of $k'$-subsets of $X \setminus S_x$. We
verify that $(X \setminus S_x, \cB')$ is a graphical
$t'$-$\bigl(\binom{n-1}{2}, k', \lambda\bigr)$ design. Let $R$ be a
$t'$-subset of $X \setminus S_x$. Then $R \cup S_x$ is a $t$-subset of $X$,
since $|R \cup S_x| = t' + (n - 1) = t$. Every $k$-subset of $X$ containing
$R \cup S_x$ contains $S_x$, so $B \mapsto B \setminus S_x$ is a bijection
between the $k$-subsets $B$ of $X$ containing $R \cup S_x$ and the
$k'$-subsets $C$ of $X \setminus S_x$ containing $R$. Hence
\[
\sum_{C \supseteq R} \mu_{\cB'}(C) = \sum_{B \supseteq R \cup S_x} \mu_{\cB}(B)
= \lambda,
\]
the last equality because $(X, \cB)$ has strength $t$. Thus, every
$t'$-subset of $X \setminus S_x$ lies in exactly $\lambda$ members of
$\cB'$, counted with multiplicity. For graphicality, let $\sigma \in \Sym(V
\setminus \{x\})$, and extend it to a permutation of $V$ by fixing $x$. The
extension fixes $S_x$ setwise, so $\sigma(C \cup S_x) = \sigma(C) \cup S_x$,
and the graphicality of $(X, \cB)$ gives
\[
\mu_{\cB'}(\sigma(C)) = \mu_{\cB}(\sigma(C) \cup S_x) = \mu_{\cB}(\sigma(C \cup S_x))
= \mu_{\cB}(C \cup S_x) = \mu_{\cB'}(C).
\]
The identities \eqref{eq:invariants} follow directly:
\[
\binom{n - 1}{2} - t'
= \binom{n}{2} - (n - 1) - t + (n - 1)
= \binom{n}{2} - t
\]
and $k' - t' = k - t$. Since $(X, \cB)$ is nontrivial, \eqref{eq:qkt} gives
$0 < d < q$. By \eqref{eq:invariants}, we have
\[
k' - t' = d > 0
\qquad \text{and} \qquad
\binom{n - 1}{2} - k' = q - d > 0,
\]
that is, $t' < k' < \binom{n-1}{2}$. The design $(X \setminus S_x, \cB')$ is
not empty, because every $t'$-subset of $X \setminus S_x$ lies in exactly
$\lambda \geq 1$ members of $\cB'$. By the definition of nontriviality, it
therefore either is nontrivial or contains the complete design. Suppose
finally that $(X \setminus S_x, \cB')$ contains the complete design. Let $R$
be a $t'$-subset of $X \setminus S_x$. Counted with multiplicity, at least
$\binom{\binom{n-1}{2} - t'}{k' - t'}$ members of $\cB'$ contain $R$, one
for each $k'$-subset containing $R$. By \eqref{eq:invariants}, this number
is $\binom{q}{d}$, so $\lambda \geq \binom{q}{d}$. Since $1 \leq d \leq q -
1$, we have $\lambda \geq \binom{q}{d} \geq \binom{q}{1} = q > t$, the last
inequality by \eqref{eq:qkt}. This proves the last assertion.
\end{proof}

Iterating the derivation gives the following reduction.

\begin{proposition}
\label{prop:descent}
Let $(X, \cB)$ be a nontrivial graphical $t$-$\bigl(\binom{n}{2}, k,
\lambda\bigr)$ design with $t \geq 2$. Then either $t < \lambda$, or there
is a nontrivial graphical $t^*$-$\bigl(\binom{n^*}{2}, k^*, \lambda\bigr)$
design with $2 \leq t^* \leq n^*$ whose parameters satisfy
\[
\binom{n^*}{2} - t^* = \binom{n}{2} - t
\qquad \text{and} \qquad
k^* - t^* = k - t.
\]
\end{proposition}

\begin{proof}
Put $q = \binom{n}{2} - t$ and $d = k - t$. Apply Lemma~\ref{lem:derivation}
repeatedly, for as long as the current design is nontrivial and its strength
exceeds its number of vertices. We call this repeated derivation the
\emph{full-star descent}. In the notation of Lemma~\ref{lem:derivation},
each derived design satisfies $t' < k' < \binom{n-1}{2}$, so it either
contains the complete design or is again nontrivial. By
\eqref{eq:invariants}, every design along the descent has the same $q$ and
the same $d$ as the given design. Since the strength strictly decreases, the
descent terminates. There are two possible outcomes.

Suppose first that, at some stage, a derived design contains the complete
design. Lemma~\ref{lem:derivation}, applied to the design from which it was
derived, gives $\lambda \geq \binom{q}{d} \geq q$, since that design has the
same $q$ and $d$ as $(X, \cB)$. By \eqref{eq:qkt} for $(X, \cB)$, we have $t
< q$. Hence $t < \lambda$.

Otherwise, the descent stops at a nontrivial graphical $t^*$-design whose
strength $t^*$ is at most its number of vertices $n^*$. We call this design
the \emph{terminal design}. Its strength is at least $2$: if no derivation
was performed, this is the hypothesis $t \geq 2$, and otherwise the last
derivation started from a design whose strength $s$ exceeded its number of
vertices $m$, so the terminal strength is $s - (m - 1) \geq 2$. The terminal
design has the same $q$ and $d$ as $(X, \cB)$, that is,
\[
\binom{n^*}{2} - t^* = \binom{n}{2} - t
\qquad \text{and} \qquad
k^* - t^* = k - t.
\qedhere
\]
\end{proof}

\begin{corollary}
\label{cor:range}
A nontrivial graphical $t$-$\bigl(\binom{n}{2}, k, \lambda\bigr)$ design
with $2 \leq t \leq n$ satisfies
\begin{equation}
\label{eq:range}
2 \leq t \leq n < 2t + \lambda + 4.
\end{equation}
\end{corollary}

\begin{proof}
Immediate from Theorem~\ref{thm:chouinard}.
\end{proof}

\section{Proof of the Finiteness Theorem}
\label{sec:mainproof}

Graphicality makes the multiplicity $\mu_{\cB}$ constant on the isomorphism
classes of spanning subgraphs, and the strength makes the blocks containing
a fixed $t$-subset few. The following lemma combines the two facts into a
lower bound on the excess.

\begin{lemma}
\label{lem:km}
Let $(X, \cB)$ be a graphical $t$-$\bigl(\binom{n}{2}, k, \lambda\bigr)$
design, let $T$ be a spanning subgraph of $K_n$ with exactly $t$ edges and
$E(T) \neq E(K_n)$, and let $L$ be the smallest cardinality of an
$\Aut(T)$-orbit on $E(K_n) \setminus E(T)$. If a block $B$ contains $E(T)$
and $k > t$, then
\[
\lambda \geq |\Aut(T) \cdot B| \geq \frac{L}{k - t},
\]
where $\Aut(T) \cdot B = \{\sigma(B) : \sigma \in \Aut(T)\}$ is the orbit of
$B$, and hence $k - t \geq L/\lambda$.
\end{lemma}

\begin{proof}
Since $T$ has exactly $t$ edges and $(X, \cB)$ has strength $t$, the members
of $\cB$ containing $E(T)$ number exactly $\lambda$, counted with
multiplicity, so at most $\lambda$ distinct blocks contain $E(T)$. For
$\sigma \in \Aut(T)$, graphicality gives $\mu_{\cB}(\sigma(B)) =
\mu_{\cB}(B) \geq 1$, and $E(T) = \sigma(E(T)) \subseteq \sigma(B)$, so
every member of $\Aut(T) \cdot B$ is a block containing $E(T)$. Hence
$|\Aut(T) \cdot B| \leq \lambda$. Let $e \in B \setminus E(T)$. Every edge
$\sigma(e)$ with $\sigma \in \Aut(T)$ lies in $\sigma(B) \setminus E(T) =
\sigma(B \setminus E(T))$, so the $\Aut(T)$-orbit of $e$ is contained in the
union of the sets $B' \setminus E(T)$, where $B'$ runs over $\Aut(T) \cdot
B$. Each of these sets has cardinality $k - t$. Hence $L \leq |\Aut(T) \cdot
B|\,(k - t)$.
\end{proof}

We seek a spanning $t$-edge graph $T$ whose automorphism group has only
orbits of cardinality $\Omega(t^2)$ on $E(K_n) \setminus E(T)$. We construct
$T$ as a disjoint union of copies of $K_1$, $K_2$, and $K_4$.

\begin{lemma}
\label{lem:testgraph}
For all sufficiently large $t$, and all $n$ with $t \leq n$, there exist
positive integers $m_1, m_2, m_4 = \Omega(t)$ such that
\[
T = m_1 K_1 \dcup m_2 K_2 \dcup m_4 K_4
\]
is a spanning $t$-edge graph on $n$ vertices.
\end{lemma}

\begin{proof}
Put
\[
r = n - t \geq 0
\qquad \text{and} \qquad
M = \left\lceil \frac{t}{7} \right\rceil,
\]
and set
\[
m_1 = r + 8M - t, \qquad m_2 = t - 6M, \qquad \text{and} \qquad m_4 = M.
\]
All three are integers. The number of edges of $T$ is
\[
m_2 + 6m_4 = (t - 6M) + 6M = t,
\]
and its number of vertices is
\begin{align*}
m_1 + 2m_2 + 4m_4 &= (r + 8M - t) + (2t - 12M) + 4M \\
&= r + t \\
&= n.
\end{align*}
For the lower bounds, $M \geq t/7$ gives
\[
m_1 = r + 8M - t \geq \frac{t}{7}
\qquad \text{and} \qquad
m_4 = M \geq \frac{t}{7},
\]
while $M \leq (t + 6)/7$ gives
\[
m_2 = t - 6M \geq \frac{t - 36}{7},
\]
which is positive for $t \geq 72$. Thus, for $t \geq 72$, the graph $T$ has
the required properties, with all three counts at least $(t - 36)/7$.
\end{proof}

Let $T = m_1 K_1 \dcup m_2 K_2 \dcup m_4 K_4$ be a spanning subgraph of
$K_n$. For each $p \in \{1, 2, 4\}$, let $C_1, \ldots, C_{m_p}$ be the
$K_p$-components of $T$, and write the vertex set of $C_i$ as $\{x_{i,1},
\ldots, x_{i,p}\}$. Each permutation $\pi$ of $\{1, \ldots, m_p\}$ defines
the element of $\Sym(V)$ that maps $x_{i,j}$ to $x_{\pi(i),j}$ for all $i$
and $j$ and fixes every vertex outside the $K_p$-components. These elements
preserve $E(T)$ and form a subgroup of $\Aut(T)$ isomorphic to the symmetric
group of degree $m_p$, which we denote $S_{m_p}$.

\begin{lemma}
\label{lem:orbitbound}
Let $T = m_1 K_1 \dcup m_2 K_2 \dcup m_4 K_4$ be a spanning subgraph of
$K_n$ with $m_1, m_2, m_4 \geq 2$. Then every $\Aut(T)$-orbit on $E(K_n)
\setminus E(T)$ has cardinality at least $\binom{m}{2}$, where $m =
\min(m_1, m_2, m_4)$.
\end{lemma}

\begin{proof}
Let $\{x, y\} \in E(K_n) \setminus E(T)$. The vertices $x$ and $y$ lie in
two distinct components of $T$, because each component is complete. Suppose
first that these components have distinct types $K_p$ and $K_q$, with $p
\neq q$. The groups $S_{m_p}$ and $S_{m_q}$ are transitive on the
$K_p$-components and on the $K_q$-components, and they move $x$ and $y$
independently. The $\Aut(T)$-orbit of $\{x, y\}$ therefore contains, for
each choice of a $K_p$-component and a $K_q$-component, an edge joining
them, and so has cardinality at least
\begin{equation}
\label{eq:orbitdistinct}
m_p m_q.
\end{equation}
Suppose now that the two components have the same type $K_p$. Since $m_p
\geq 2$, the group $S_{m_p}$ is $2$-transitive on the $K_p$-components. The
$\Aut(T)$-orbit of $\{x, y\}$ therefore contains, for each pair of distinct
$K_p$-components, an edge joining them, and so has cardinality at least
\begin{equation}
\label{eq:orbitsame}
\binom{m_p}{2}.
\end{equation}
Both \eqref{eq:orbitdistinct} and \eqref{eq:orbitsame} are at least
$\binom{m}{2}$, since $m_p m_q \geq m^2$. Since $\{x, y\}$ was arbitrary,
every $\Aut(T)$-orbit on $E(K_n) \setminus E(T)$ has cardinality at least
$\binom{m}{2}$.
\end{proof}

Finally, we bound the block sizes of the designs satisfying
\eqref{eq:range}.

\begin{proposition}
\label{prop:excess}
Every nontrivial graphical $t$-$\bigl(\binom{n}{2}, k, \lambda\bigr)$ design
satisfying \eqref{eq:range} has
\begin{equation}
\label{eq:kbound}
k \leq e \lambda^{1/t} v^{1 - s/t} s^{s/t},
\qquad \text{where} \qquad
v = \binom{n}{2}
\quad \text{and} \quad
s = \left\lfloor \frac{t}{2} \right\rfloor.
\end{equation}
In particular, $k = O(t^{3/2})$.
\end{proposition}

\begin{proof}
The design has $v = \binom{n}{2}$ points, and $s \geq 1$ because $t \geq 2$
by \eqref{eq:range}. Lemma~\ref{lem:vkt} gives $v > k + t \geq k + s$, and
Theorem~\ref{thm:rcw} yields $b \geq \binom{v}{s}$. Combining this with
\eqref{eq:blocks}, we obtain
\[
\lambda \frac{\binom{v}{t}}{\binom{k}{t}} \geq \binom{v}{s}.
\]
The elementary estimates
\[
\binom{v}{t} \leq \left( \frac{ev}{t} \right)^t,
\qquad
\binom{k}{t} \geq \left( \frac{k}{t} \right)^t,
\qquad \text{and} \qquad
\binom{v}{s} \geq \left( \frac{v}{s} \right)^s
\]
give
\[
\lambda \left( \frac{ev}{k} \right)^t \geq \left( \frac{v}{s} \right)^s.
\]
Solving for $k$ gives \eqref{eq:kbound}. By \eqref{eq:range}, $v =
\binom{n}{2} = \Theta(t^2)$. Since $s \leq t/2$ and $1 - s/t \leq
\frac{1}{2} + \frac{1}{2t}$, while $v^{1/(2t)}$ and $\lambda^{1/t}$ are
$O(1)$, the right-hand side of \eqref{eq:kbound} is $O\bigl(v^{1/2}
t^{1/2}\bigr) = O(t^{3/2})$.
\end{proof}

\begin{proof}[Proof of Theorem~\ref{thm:main}]
We first bound the strengths of the nontrivial graphical $t$-designs of
index $\lambda$ satisfying \eqref{eq:range}. Consider such a design with $t$
sufficiently large. Since $t \leq n$, Lemma~\ref{lem:testgraph} provides a
spanning $t$-edge graph $T = m_1 K_1 \dcup m_2 K_2 \dcup m_4 K_4$ with $m_1,
m_2, m_4 = \Omega(t)$, and Lemma~\ref{lem:orbitbound} shows that every
$\Aut(T)$-orbit on $E(K_n) \setminus E(T)$ has cardinality at least
$\binom{m}{2} = \Omega(t^2)$, where $m = \min(m_1, m_2, m_4)$. Since $(X,
\cB)$ has strength $t$ and $T$ has $t$ edges, some block $B$ contains
$E(T)$, and $k > t$ by nontriviality. By Lemma~\ref{lem:km} and
Proposition~\ref{prop:excess},
\[
\lambda \geq |\Aut(T) \cdot B| \geq \frac{\binom{m}{2}}{k - t}
= \frac{\Omega(t^2)}{O(t^{3/2})} = \Omega(t^{1/2}).
\]
The orbit of $B$ thus grows without bound while the design condition caps
its cardinality at $\lambda$, which is impossible for all sufficiently large
$t$. Hence there is a bound $S_\lambda$, depending only on $\lambda$, such
that every nontrivial graphical $t$-design of index $\lambda$ satisfying
\eqref{eq:range} has $t < S_\lambda$.

It remains to transfer this bound from the range \eqref{eq:range} to an
arbitrary nontrivial graphical $t$-design. Consider a nontrivial graphical
$t$-design $(X, \cB)$ of index $\lambda$ with $t \geq 2$ that does not
satisfy \eqref{eq:range}. By Proposition~\ref{prop:descent}, either $t <
\lambda$, or there is a nontrivial graphical $t^*$-$\bigl(\binom{n^*}{2},
k^*, \lambda\bigr)$ design with $2 \leq t^* \leq n^*$ and $\binom{n^*}{2} -
t^* = \binom{n}{2} - t$. In the second case, that design satisfies
\eqref{eq:range} by Corollary~\ref{cor:range}, so $t^* < S_\lambda$, and
Theorem~\ref{thm:chouinard} gives $n^* < 2t^* + \lambda + 4$. Both $t^*$ and
$n^*$ are therefore bounded in terms of $\lambda$, and hence so is
$\binom{n^*}{2} - t^*$, which equals $\binom{n}{2} - t$ by
Proposition~\ref{prop:descent}. Since $t < \binom{n}{2} - t$ by
\eqref{eq:qkt}, the strength $t$ is bounded in terms of $\lambda$ in this
case as well.

The strength of every nontrivial graphical $t$-design of index $\lambda$
with $t \geq 2$ is therefore bounded in terms of $\lambda$.
Theorem~\ref{thm:chouinard} then bounds $n$, and $k < \binom{n}{2}$ bounds
$k$, so only finitely many parameter quadruples $(n, t, k, \lambda)$ occur.
For each fixed $n$ and $k$, a $t$-design on $E(K_n)$ is determined by its
multiplicity function
\[
\mu_{\cB} : \binom{E(K_n)}{k} \longrightarrow \{0, 1, \ldots, \lambda\},
\]
where $\mu_{\cB}(B) \leq \lambda$ because any $t$-subset of $B$ lies in
exactly $\lambda$ members of $\cB$, of which $B$ accounts for
$\mu_{\cB}(B)$. There are only finitely many such functions. The result
follows.
\end{proof}

\section{Explicit Bounds}
\label{sec:explicitsec}

Every quantity that the proof of Theorem~\ref{thm:main} bounds
asymptotically is now bounded numerically. Numerical constants alone,
however, would retain the factor $\lambda$ of Lemma~\ref{lem:km}. To obtain
a sufficiently strong bound on the strengths of the designs satisfying
\eqref{eq:range}, we sharpen Lemma~\ref{lem:km}: a large alternating
subgroup of $\Aut(T)$ fixes every block containing $E(T)$ setwise, so $B
\setminus E(T)$ is a union of large orbits, and the factor $\lambda$
disappears from the lower bound on $k - t$. We begin with the counts of the
test graph, which the proof of Lemma~\ref{lem:testgraph} already provides.

\begin{lemma}
\label{lem:testgraph2}
For $72 \leq t \leq n$, the graph $T$ of Lemma~\ref{lem:testgraph} may be
chosen with
\[
m_1, m_4 \geq \frac{t}{7}
\qquad \text{and} \qquad
m_2 \geq \frac{t - 36}{7}.
\]
\end{lemma}

\begin{proof}
The choice $M = \lceil t/7 \rceil$ in the proof of Lemma~\ref{lem:testgraph}
gives exactly these bounds for $t \geq 72$.
\end{proof}

We next isolate the numerical effect of the full-star descent: a bound $S$
on the strengths of the designs satisfying \eqref{eq:range} yields explicit
bounds on all the parameters of all designs.

\begin{proposition}
\label{prop:transfer}
Let $S \geq 7\lambda + 43$, and suppose that every nontrivial graphical
$t$-design of index $\lambda$ satisfying \eqref{eq:range} has $t < S$. Then
every nontrivial graphical $t$-$\bigl(\binom{n}{2}, k, \lambda\bigr)$ design
with $t \geq 2$ satisfies
\[
t < 2.3\, S^2,
\qquad
n < 4.7\, S^2,
\qquad \text{and} \qquad
k < 11.1\, S^4.
\]
\end{proposition}

\begin{proof}
A design that satisfies \eqref{eq:range} has $t < S \leq 2.3\, S^2$.
Consider a nontrivial graphical $t$-design $(X, \cB)$ of index $\lambda$
with $t \geq 2$ that does not satisfy \eqref{eq:range}, and apply the
descent of Proposition~\ref{prop:descent}. If, at some stage, the derived
design contains the complete design, then $t < \lambda < S$, as in the proof
of Proposition~\ref{prop:descent}. Otherwise, the terminal design satisfies
\eqref{eq:range} by Corollary~\ref{cor:range}, so its parameters satisfy
$t^* < S$ and
\[
n^* < 2t^* + \lambda + 4 \leq 2S + \lambda + 2 \leq \frac{15S}{7},
\]
using $\lambda + 2 \leq S/7$. The quantity $q = \binom{n}{2} - t$ of $(X,
\cB)$ is invariant under the descent, so $q = \binom{n^*}{2} - t^* <
\binom{n^*}{2}$, and \eqref{eq:qkt}, applied to $(X, \cB)$ itself, gives $t
< q$. Hence
\[
t < q < \binom{n^*}{2} < \frac{(15S/7)^2}{2} < 2.3\, S^2.
\]
Thus, $t < 2.3\, S^2$ in all cases. Theorem~\ref{thm:chouinard} then gives
\[
n < 2t + \lambda + 4 < 4.6\, S^2 + \frac{S}{7} < 4.7\, S^2,
\]
and $k < \binom{n}{2} < n^2/2 < 11.1\, S^4$.
\end{proof}

The next step replaces the asymptotic conclusion of
Proposition~\ref{prop:excess} by a numerical one.

\begin{proposition}
\label{prop:excess2}
Every nontrivial graphical $t$-$\bigl(\binom{n}{2}, k, \lambda\bigr)$ design
satisfying \eqref{eq:range} with $t \geq T_\lambda = \max(10^5,\ 7\lambda +
43)$ has
\[
k < 2.92\, t^{3/2}.
\]
\end{proposition}

\begin{proof}
Put $u = 2t + \lambda + 4$. Since $t \geq 7\lambda + 43$, we have $\lambda +
4 \leq (t - 15)/7$ and hence $u \leq 15t/7$. Also, $n < u$ by
\eqref{eq:range}, so $v < u^2/2$. Substituting the bounds $s \leq t/2$, $1 -
s/t \leq \frac{1}{2} + \frac{1}{2t}$, $v < u^2/2$, and $\lambda < u$ into
\eqref{eq:kbound} gives
\[
k \leq e \lambda^{1/t} v^{1 - s/t} s^{s/t}
< e\, u^{2/t}\, \frac{u \sqrt{t}}{2}.
\]
The function $2\ln(15t/7)/t$ is decreasing for $t \geq 2$, and its value at
$t = 10^5$ is less than $0.000246$. For $t \geq 10^5$, we have
\[
u^{2/t} \leq \exp\left( \frac{2 \ln(15t/7)}{t} \right)
\leq e^{0.000246} < 1.0003,
\]
and hence
\[
k < e \cdot 1.0003 \cdot \frac{15}{14}\, t^{3/2} < 2.92\, t^{3/2}.
\qedhere
\]
\end{proof}

The orbit argument of Lemma~\ref{lem:km} loses a factor $\lambda$. We now
remove this loss by showing that a large subgroup of $\Aut(T)$ fixes every
block containing $E(T)$ setwise. This rests on the following lemma, the
specialization to $t$-designs of \cite[Lemma 1.2]{Chouinard1996}, a
corollary of \cite[Theorem 2.1]{CKK1983} that holds for repeated blocks.

\begin{lemma}[Chouinard]
\label{lem:blockfix}
Let $(X, \cB)$ be a graphical $t$-$\bigl(\binom{n}{2}, k, \lambda\bigr)$
design, let $T_0$ be a $t$-subset of $X$, and let $G \leq \Sym(V)$ fix $T_0$
setwise. Then every block $B$ of $\cB$ with $T_0 \subseteq B$ is fixed
setwise by a subgroup of $G$ of index at most $\lambda$.
\end{lemma}

For a graph $T$ as in Section~\ref{sec:mainproof} and $p \in \{1, 2, 4\}$,
let $A_{m_p} \leq S_{m_p}$ denote the alternating subgroup.

\begin{lemma}
\label{lem:orbit}
Let $(X, \cB)$ be a graphical $t$-$\bigl(\binom{n}{2}, k, \lambda\bigr)$
design, and let $T = m_1 K_1 \dcup m_2 K_2 \dcup m_4 K_4$ be a spanning
$t$-edge subgraph of $K_n$ with $m_p \geq \max\{6, \lambda + 1\}$ for each
$p \in \{1, 2, 4\}$. Then there is a subgroup $H \leq \Aut(T)$ that fixes
every block containing $E(T)$ setwise and contains $A_{m_p}$ for each $p \in
\{1, 2, 4\}$.
\end{lemma}

\begin{proof}
Put $H = \langle A_{m_1}, A_{m_2}, A_{m_4} \rangle \leq \Aut(T)$, and let $p
\in \{1, 2, 4\}$ and $B$ be a block with $E(T) \subseteq B$. The group
$A_{m_p}$ fixes the $t$-subset $E(T)$ of $X$ setwise, so, by
Lemma~\ref{lem:blockfix}, $B$ is fixed setwise by a subgroup $U \leq
A_{m_p}$ of index at most $\lambda$. Suppose that $U$ is a proper subgroup,
say of index $j$, so that $2 \leq j \leq \lambda$. The action of $A_{m_p}$
on the $j$ cosets of $U$ gives a homomorphism from $A_{m_p}$ to the
symmetric group of degree $j$. Since $m_p \geq 6$, the group $A_{m_p}$ is
simple, so the kernel of this homomorphism, a normal subgroup contained in
$U$, is trivial. Thus $A_{m_p}$ embeds in the symmetric group of degree $j$,
so
\[
\frac{m_p!}{2} = |A_{m_p}| \leq j! \leq \lambda! \leq (m_p - 1)!,
\]
contradicting $m_p!/2 = \frac{m_p}{2}\,(m_p - 1)! > (m_p - 1)!$. Hence $U =
A_{m_p}$, that is, $A_{m_p}$ fixes $B$ setwise. Since $p$ and $B$ were
arbitrary, $H$ fixes every block containing $E(T)$ setwise.
\end{proof}

\begin{lemma}[Orbit gap]
\label{lem:gap}
Let $T$ be a spanning subgraph of $K_n$ with $E(T) \neq E(K_n)$, let $B$ be
a $k$-subset of $E(K_n)$ with $E(T) \subseteq B$, and let $H \leq \Sym(V)$
fix both $E(T)$ and $B$ setwise. Let $L$ be the smallest cardinality of an
$H$-orbit on $E(K_n) \setminus E(T)$. Then $k - |E(T)|$ is either zero or at
least $L$.
\end{lemma}

\begin{proof}
The set $B \setminus E(T)$ is $H$-invariant, hence a union of $H$-orbits,
each of cardinality at least $L$.
\end{proof}

\begin{theorem}
\label{thm:explicit}
Every nontrivial graphical $t$-$\bigl(\binom{n}{2}, k, \lambda\bigr)$ design
with $t \geq 2$ satisfies
\[
t < 2.3\, T_\lambda^2,
\qquad
n < 4.7\, T_\lambda^2,
\qquad \text{and} \qquad
k < 11.1\, T_\lambda^4,
\]
where $T_\lambda = \max(10^5,\ 7\lambda + 43)$, as in
Proposition~\ref{prop:excess2}.
\end{theorem}

\begin{proof}
Consider a nontrivial graphical $t$-design satisfying \eqref{eq:range} with
$t \geq T_\lambda$. Proposition~\ref{prop:excess2} gives $k < 2.92\,
t^{3/2}$. Since $t \leq n$, Lemma~\ref{lem:testgraph2} provides a spanning
$t$-edge graph $T = m_1 K_1 \dcup m_2 K_2 \dcup m_4 K_4$ with $m_1, m_2, m_4
\geq (t - 36)/7$, and $t \geq 78$ gives $(t - 36)/7 \geq 6$, while $t \geq
7\lambda + 43$ gives $(t - 36)/7 \geq \lambda + 1$. Hence
Lemma~\ref{lem:orbit} provides a subgroup $H \leq \Aut(T)$ that fixes every
block containing $E(T)$ setwise and contains $A_{m_p}$ for each $p \in \{1,
2, 4\}$. The proof of Lemma~\ref{lem:orbitbound} used only that the groups
$S_{m_p}$ permute the components of each type $2$-transitively. Since $m_p
\geq 6$, the alternating groups $A_{m_p}$ do so as well, so every $H$-orbit
on $E(K_n) \setminus E(T)$ has cardinality at least $\binom{m}{2} \geq (t -
36)(t - 43)/98$, where $m = \min(m_1, m_2, m_4)$. Since $(X, \cB)$ has
strength $t$ and $T$ has $t$ edges, some block $B$ contains $E(T)$. The
group $H$ fixes both $E(T)$ and $B$ setwise, so Lemma~\ref{lem:gap} gives
either $k = t$, which contradicts the nontriviality condition $t < k$, or
\[
k - t \geq \frac{(t - 36)(t - 43)}{98}.
\]
Since $t \geq 10^5$, we have $(t - 36)(t - 43) \geq 315\, t^{3/2}$, so $k -
t \geq (315/98)\, t^{3/2} > 3.2\, t^{3/2}$, which contradicts $k < 2.92\,
t^{3/2}$. Hence, every nontrivial graphical $t$-design satisfying
\eqref{eq:range} has $t < T_\lambda$. Since $T_\lambda \geq 7\lambda + 43$,
Proposition~\ref{prop:transfer} completes the proof.
\end{proof}

\section{Concluding Remarks}
\label{sec:consequences}

For large $\lambda$, the bounds of Theorem~\ref{thm:explicit} are of order
$\lambda^2$ for the strength and the number of vertices, and of order
$\lambda^4$ for the block size. The content lies in the bounds on $t$ and
$n$, since the bound on $k$ is simply $k < \binom{n}{2}$.

The constants in Theorem~\ref{thm:explicit} come from the choice of $K_4$ in
Lemma~\ref{lem:testgraph}. Let $p \geq 4$. If the test graph is constructed
instead as a disjoint union
\[
T = m_1 K_1 \dcup m_2 K_2 \dcup m_p K_p,
\]
subject to $m_2 + \binom{p}{2} m_p = t$ and $m_1 + 2 m_2 + p\, m_p = n$,
then, for $t \leq n$ and $t \geq t_0(p)$, the argument of
Lemma~\ref{lem:testgraph} yields
\[
m_1, m_2, m_p \geq c_p t - a_p,
\qquad \text{where} \qquad
c_p = \frac{\binom{p}{2} - p}{3 \binom{p}{2} - p},
\]
with $a_p$ and $t_0(p) \geq p$ depending only on $p$. No choice of the three
counts exceeds $c_p t$ when $n = t$. The condition $n \geq p$, necessary for
$T$ to fit on $n$ vertices, follows from $n \geq t \geq t_0(p)$. Here $c_4 =
1/7$, and $c_p \to 1/3$ as $p \to \infty$. The orbit bound of
Lemma~\ref{lem:orbitbound} becomes $\binom{c_p t - a_p}{2}$, which
approaches $t^2/18$. Carrying these constants through the proof of
Theorem~\ref{thm:explicit} shows that the strength of a nontrivial graphical
$t$-design satisfying \eqref{eq:range} is less than $(3 +
\varepsilon)\lambda$ once $p$ and $\lambda$ are large enough in terms of
$\varepsilon$, and the descent of Proposition~\ref{prop:descent} then shows
that, for every $\varepsilon > 0$, there are $p(\varepsilon)$ and
$\lambda_0(\varepsilon)$ such that every nontrivial graphical
$t$-$\bigl(\binom{n}{2}, k, \lambda\bigr)$ design with $t \geq 2$ and
$\lambda \geq \lambda_0(\varepsilon)$ satisfies
\[
t < \Bigl(\frac{49}{2} + \varepsilon\Bigr)\lambda^2,
\qquad
n < (49 + \varepsilon)\lambda^2,
\qquad \text{and} \qquad
k < \Bigl(\frac{2401}{2} + \varepsilon\Bigr)\lambda^4.
\]

The order of these bounds is unlikely to be sharp. The determination of the
graphical $t$-wise balanced designs of index at most two in \cite{CKK1983}
produced only designs of small strength, and no family of nontrivial
graphical $t$-designs whose strength grows with the index is known. Whether
the strength of a nontrivial graphical $t$-design can grow even linearly in
$\lambda$ is an open question. So is Chouinard's conjecture for simple
graphical $t$-wise balanced designs \cite{Chouinard1996}.

\section*{Acknowledgement}

Research supported by SUTD Grant SKI 2021\_07\_04.

\section*{Use of Artificial Intelligence}

Following the recommendations of the Leiden Declaration on Artificial
Intelligence and Mathematics \cite{Leiden2026}, the author discloses the use
of automated tools in the preparation of this paper. The results and their
proofs are due to the author. A large language model (Claude, Anthropic) was
used, under the author's direction, to draft and revise the exposition, to
reorganize the manuscript, to suggest simplifications of the proofs, to
typeset and compile the manuscript, to check the numerical constants in
Proposition~\ref{prop:excess2} and Theorem~\ref{thm:explicit}, to compare
the statements of the results quoted from \cite{CKK1983, Chouinard1996,
RCW1975} with their original texts, and to carry out the routine
computations in the remarks of Section~\ref{sec:consequences}. Every
suggestion of the model was reviewed and verified by the author, who takes
sole responsibility for the content of this paper.

\end{document}